\documentclass[12pt]{amsart}
\usepackage{latexsym}
\usepackage{amsmath}
\usepackage{amsfonts}
\usepackage{amsthm}
\usepackage{amssymb}
\usepackage{ifthen}
\usepackage{enumerate}
\usepackage[T1]{fontenc}
\usepackage{dutchcal}
\usepackage{cite}

\usepackage[mathscr]{eucal}
\newcommand*{\rom}[1]{\expandafter\@slowromancap\romannumeral #1@}

\DeclareFontFamily{U}{mathx}{}
\DeclareFontShape{U}{mathx}{m}{n}{<-> mathx10}{}
\DeclareSymbolFont{mathx}{U}{mathx}{m}{n}
\DeclareMathAccent{\widehat}{0}{mathx}{"70}
\DeclareMathAccent{\widecheck}{0}{mathx}{"71}
\def\eq#1{{\rm(\ref{E#1})}}
\def\Eq#1#2{\ifthenelse{\equal{#1}{*}}
  {\begin{equation*}\begin{aligned}#2\end{aligned}\end{equation*}}
  {\begin{equation}\begin{aligned}\label{E#1}#2\end{aligned}\end{equation}}}

\newcounter{allenv}[section]

\newtheorem{thm}{Theorem}
\newtheorem*{thm*}{Theorem}
\newtheorem{prop}{Proposition}
\newtheorem*{conj*}{Conjecture}
\newtheorem{coro}{Corollary}

\theoremstyle{remark}
\newtheorem*{remark*}{Remark}

\newtheorem*{exmp*}{Example}

\theoremstyle{definition}

\author[Tibor Kiss]{Tibor Kiss}

\title[]{A Counterexample to Matkowski's Conjecture for Quasi Graph-Additive Functions} 

\address{Institute of Mathematics,
University of Debrecen,
4002 Debrecen, Pf.~400, Hungary}
\email{kiss.tibor@science.unideb.hu}



\thanks{The author’s research was supported by the HUN‑REN Hungarian Research Network.}

\def\eq#1{{\rm(\ref{E#1})}}

\def\Eq#1#2{\ifthenelse{\equal{#1}{*}}
	{\begin{equation*}\begin{aligned}[]#2\end{aligned}\end{equation*}}
	{\begin{equation}\begin{aligned}\label{E#1}#2\end{aligned}\end{equation}}}

\begin{document}

\begin{abstract}
In this paper we investigate a conjecture of Janusz Matkowski concerning the continuous solutions of the functional equation
\Eq{*}{
f\big(f(-x)+x\big)=f\big(-f(x)\big)+f(x),\qquad x\in\mathbb{R}.
}

Matkowski conjectured that all continuous solutions must necessarily be linear on both the negative and the positive half-line. We show, however, that the family of continuous solutions to the equation in question is far richer than anticipated: there exist continuous solutions that admit an arbitrary part. 

In addition, we provide a sufficient condition which, in the continuous setting, enforces the conclusion predicted by Matkowski’s Conjecture.
\end{abstract}

\maketitle

\section{Introduction}

On 18 September 2025, on the fourth day of the \emph{21st International Conference on Functional Equations and Inequalities} in Będlewo, Dorota Głazowska delivered an excellent and engaging plenary lecture on weakly associative functions and means. In her talk, she mentioned a result, which, in the family of translative binary operations, characterizes weak associativity in terms of a functional equation. She also presented a conjecture concerning the continuous solutions of this functional equation. The result in question and the conjecture originally appeared in the paper \cite{Mat25} of Janusz Matkowski. In order to formulate them, first we introduce some basic concepts.

As it is well known, a binary operation $F:\mathbb{R}\times\mathbb{R}\to\mathbb{R}$ is called \emph{associative} if
\Eq{*}{
F\big(F(x,y),z\big)=F\big(x,F(y,z)\big)
}
holds for any choice of $x,y,z\in\mathbb{R}$. This property is frequently studied in the theory of means as well. It is easy to see that any quasi‑sum mean is associative. In contrast, strictly weighted quasi‑arithmetic means never satisfy the above equality. In her talk, Głazowska addressed the question of how one might weaken the associativity equation so that strictly weighted quasi‑arithmetic means would also satisfy it. This leads to the concept of \emph{weak associativity}, namely to the identity
\Eq{*}{
F\big(F(x,y),x\big)=F\big(x,F(y,x)\big),\qquad x,y\in\mathbb{R}.
}
It turned out that every strict weighted quasi‑arithmetic mean is weakly associative. The interested reader can find further information in this direction in \cite{GlaMat25}.

Now we turn to the direction initiated by Matkowski in his paper \cite{Mat25}. A binary operation $F:\mathbb{R}\times\mathbb{R}\to\mathbb{R}$ is called \emph{translative} if
\Eq{*}{
F(x+z,y+z)=F(x,y)+z
}
holds for all $x,y,z\in\mathbb{R}$. Then the characterization theorem can be stated as follows.

\begin{thm*}[J. Matkowski, \cite{Mat25}] Let $F:\mathbb{R}\times\mathbb{R}\to\mathbb{R}$ be translative and $f:\mathbb{R}\to\mathbb{R}$ be defined as $f(x):=F(x,0)$. The binary operation $F$ is weakly associative if and only if
\Eq{kadd}{
f\big(f(-x)+x\big)=f\big(-f(x)\big)+f(x)
}
is satisfied for all $x\in\mathbb{R}$.
\end{thm*}

Equations of this kind have been extensively investigated by numerous authors since the 1970s. Related results can be found, for instance, in \cite{Zdu72}, \cite{For84}, and \cite{Sab85}. A widely studied variant is
\Eq{add}{
f\big(f(x)+x\big)=f\big(f(x)\big)+f(x),\qquad x\in\mathbb{R},
}
which essentially expresses that $f$ is additive on its own graph. G. L. Forti and J. Matkowski obtained the same results concerning the solution of \eq{add} under various regularity assumptions in \cite{For83}, \cite{Mat85} and \cite{Mat93}. The most general result was obtained by Witold Jarczyk, who solved the equation assuming only continuity. The details of his far‑from‑trivial approach can be found in \cite{Jar88} and \cite{Jar91}.

\begin{thm*}[W. Jarczyk, \cite{Jar88}]
If $f:\mathbb{R}\to\mathbb{R}$ is a continuous solution of equation \eq{add} then either there exist non-negative numbers $c_-$ and $c_+$ such that
\Eq{*}{
f(x)=
\begin{cases}
c_-x&\text{if }x\leq 0,\\
c_+x&\text{if }x> 0,
\end{cases}
}
or there exists a negative number $c$ such that
\Eq{*}{
f(x)=cx,\qquad x\in\mathbb{R}.
}\end{thm*}

Motivated by this result, Matkowski, in his paper \cite{Mat25}, formulated the following conjecture.

\begin{conj*}[J. Matkowski, Remark 8. and Conjecture 2. of \cite{Mat25}]\label{cm}
Let $f:\mathbb{R}\to\mathbb{R}$ be continuous. The function $f$ solves equation \eq{kadd} if and only if there exist $a,b\in\mathbb{R}$ with $a\neq b$ such that
\Eq{*}{
f(x)=
\begin{cases}
ax&\text{if }x\leq 0,\\
bx&\text{if }x>0,
\end{cases}
}
where $a,b\in[0,1]$ or $a+b=1$.
\end{conj*}

Although Matkowski’s conjecture was well‑founded, it turned out that the family of solutions of \eq{kadd} is much richer. Equation \eq{kadd}, even in the continuous case, admits solutions that possess a nonlinear component.

\section{Solutions of \eq{kadd} with non-linear components}

In all that follows we shall assume that $f$ is known on the non-positive half‑line, and indeed that it is linear there. More precisely, we assume that there exists $a\in\mathbb{R}$ such that
\Eq{neg}{
f(x)=ax,\qquad x\leq 0.
}

This is not a serious restriction in the following sense. One can easily deduce that if $f$ is a solution of \eq{kadd}, then $x\mapsto-f(-x)$ is a solution as well. As a consequence, each of our results has a corresponding counterpart, its dual, namely the case when the function is linear on the non‑negative half‑line. As the required modifications are self‑evident, we do not describe these consequences separately.

We also take into account the origin of our functional equation, namely that it arose as a restriction of a two‑variable function. In order to obtain further conditions on $f$, we assume that $F:\mathbb{R}\times\mathbb{R}\rightarrow \mathbb{R}$ is a \emph{mean}, that is
\Eq{mean}{
\min(x,y)\leq F(x,y)\leq\max(x,y)
}
holds for all $x,y\in\mathbb{R}$. Because of the translativity of $F$, for any $x,y\in\mathbb{R}$, we can write
\Eq{*}{
F(x,y)=F(x-y,0)+y=f(x-y)+y,}
thus, in the light of \eq{mean}, it follows that
\Eq{BT}{
\min(x,0)\leq f(x)\leq \max(x,0),\qquad x\in\mathbb{R}.
}
We will refer to condition \eq{BT}, from now on, as the \emph{bow‑tie condition}. We also note that this forces, in \eq{neg}, that $0\leq a\leq 1$. Furthermore, by the squeeze theorem, for any function satisfying \eq{BT}, we must have $f(0-)=f(0+)=0$.  

We are now in a position to formulate our first statement.

\begin{prop}\label{hom}
Assume that $f:\mathbb{R}\to\mathbb{R}$ satisfies condition \eq{neg} and that $0\leq f(x)\leq x$ for all $x\geq 0$. Then $f$ solves the functional equation \eq{kadd} if and only if it is $\lambda$-homogeneous on the non-negative half-line with $\lambda\in\{a,1-a\}$.
\end{prop}

\begin{proof}
Assume first that $f$ solves equation \eq{kadd}. If $u\leq 0$, then $f(-u)+u\leq 0$, hence the left and right hand side of \eq{kadd} reduce to $af(-u)+au$ and $f(-au)+au$, respectively. By equating them, simplifying and putting $x=-u$, we obtain that \Eq{*}{f(ax)=af(x),\qquad x\geq 0.}

If $x\geq 0$, then $-x\leq 0$ and $f(x)\geq 0$. Thus, after a computation similar to the preceding one, we obtain that
\Eq{*}{f\big((1-a)x\big)=(1-a)f(x),\qquad x\geq 0.}

Assume now that $f$ is $\lambda$-homogeneous on the non-negative half-line. We show that $f$ solves equation \eq{kadd}. If $x\geq 0$, then we have that
\Eq{*}{
f\big(f(-x)+x\big)=f(-ax+x)=(1-a)f(x)=-af(x)+f(x)=f\big(-f(x)\big)+f(x).
}

If $x<0$, then
\Eq{*}{
f\big(-f(x)\big)+f(x)=f(-ax)+ax=af(-x)+ax=f\big(f(-x)+x\big),
}which finishes the proof.
\end{proof}

The above proposition shows that the cases $a=0$ and $a=1$ are exceptional. This yields the following direct consequence.

\begin{coro}
For any $g:\mathbb{R}\to\mathbb{R}$ that satisfies \eq{BT}, the function
\Eq{*}{
f:\mathbb{R}\to\mathbb{R},\qquad 
f(x):=
\begin{cases}
ax&\text{if }x\leq 0,\\
g(x)&\text{if }x>0
\end{cases}
}
solves functional equation \eq{kadd}, provided that $a=0$ or $a=1$. Moreover, if $g$ is continuous on the non-negative half-line, then $f$ is a continuous solution of \eq{kadd}.
\end{coro}

\begin{proof}
The proof is straightforward and is therefore left to the reader.
\end{proof}

The above corollary therefore shows that the Conjecture is false. From now on we deal with the strict condition $0<a<1$. Here we are going to distinguish two cases, depending on whether the ratio $\frac{\ln a}{\ln(1-a)}$ is rational or not. 

In the rational case we again obtain that the solution need not be linear over the positive half-line. The irrational case, however, forces linearity, and thus in this situation we obtain solutions of Jarczyk type (see \cite{Jar88}). 

We note that, in view of Baker's Theorem \cite[Theorem 1.4]{Bak90}, $\frac{\ln a}{\ln(1-a)}$ can be rational only if both of $\ln a$ and $\ln(1-a)$ are irrational.

\begin{prop}\label{g}
Let $\alpha,\beta\in\mathbb{R}\setminus\{-1,0,1\}$ with $\frac{\ln|\alpha|}{\ln|\beta|}\in\mathbb{Q}$. Then there exists a continuous function $g:\mathbb{R}\to\mathbb{R}$ satisfying the bow-tie condition \eq{BT}, which is $\lambda$-homogeneous for $\lambda\in\{\alpha,\beta\}$ and for which the ratio $\frac{g}{\mathrm{id}}$ is not constant over the negative and positive half-line.
\end{prop}

\begin{proof}
If $\frac{\ln|\alpha|}{\ln|\beta|}$ is rational, then there exist $n,m\in\mathbb{Z}$ with $nm\neq 0$ such that $\frac{\ln|\alpha|}{n}=\frac{\ln|\beta|}{m}$ holds. Let us denote the common value by $\gamma$. Then let $h:\mathbb{R}\to\mathbb{R}$ be a nowhere constant, continuous function with $0\leq h(x)\leq 1$ for $x\in\mathbb{R}$ that is periodic with period $\gamma$, and define
\Eq{defg}{
g:\mathbb{R}\to\mathbb{R},\qquad
g(x):=
\begin{cases}
x\cdot h(\ln |x|)&\text{if }x\neq 0,\\
0&\text{if }x=0.
\end{cases}
}

Now we prove that $g$ satisfies the properties stated in our proposition. Obviously, $\frac{g}{\mathrm{id}}$ is not constant over the negative or the positive half-line. The continuity of $g$ on the positive and negative half‑line is immediate. Due to the chain of inequalities we have for $h$, it is easy to see that $g$ satisfies the bow-tie condition \eq{BT}. Thus, by the squeeze theorem, $g$ is continuous at $0$ as well. Finally, for $x\in\mathbb{R}$, we can write that
\Eq{*}{
g(\lambda x)
=\lambda x\cdot h(\ln|\lambda x|)
=\lambda x\cdot h(\ln|x|+\ln|\lambda|)
=\lambda x\cdot h(\ln|x|+k\gamma)=\lambda g(x),
}
where $k=n$ or $k=m$ if $\lambda=\alpha$ or $\lambda=\beta$, respectively.
\end{proof}

\begin{remark*}
The function $h$ appearing in the proof of Proposition \ref{g}. does clearly exist. Indeed, say the function
\Eq{*}{
h(x):=|\sin\big(\tfrac{\pi}{\gamma}\cdot x\big)|,\qquad x\in\mathbb{R},
}
possesses all the properties required in the proof.
\end{remark*}

\begin{coro}
There exists a solution $f:\mathbb{R}\to\mathbb{R}$ to the functional equation \eq{kadd} that satisfies \eq{BT} and \eq{neg} for some $0<a<1$, and that is not linear on the non‑negative half‑line, provided that $\frac{\ln a}{\ln(1-a)}\in\mathbb{Q}$.
\end{coro}

\begin{proof}
If $0<a<1$ such that $\frac{\ln a}{\ln(1-a)}$ is rational, then, by Proposition \ref{g}., there exists a continuous function $g:\mathbb{R}\to\mathbb{R}$ which satisfies the bow-tie condition \eq{BT}, is homogeneous with $a$ and $1-a$ and for which the ratio $\frac{g}{\mathrm{id}}$ is not constant on the positive half-line. Thus, in view of Proposition \ref{hom}., the function
\Eq{*}{
f:\mathbb{R}\to\mathbb{R},\qquad
f(x):=
\begin{cases}
ax&\text{if }x\leq 0,\\
g(x)&\text{if }x>0
\end{cases}
}
is a desired solution of \eq{kadd}.
\end{proof}

Now we turn to the irrational case. As mentioned earlier, this condition will force the appearance of Jarczyk‑type solutions. We also note that the continuity of the function plays a central role in the proof below.

\begin{thm}
Assume that $f:\mathbb{R}\to\mathbb{R}$ is continuous and it satisfies the bow-tie condition \eq{BT} and \eq{neg} for some $0<a<1$ with $\frac{\ln a}{\ln(1-a)}\in\mathbb{R}\setminus\mathbb{Q}$. Then the function $f$ solves the functional equation \eq{kadd} if and only if there exists $0\leq b\leq 1$ such that
\Eq{+}{f(x)=bx,\qquad x\geq 0.}
\end{thm}

\begin{proof}
If \eq{+} is valid, then $f$ is homogeneous on the non-negative half-line. Thus, in the light of Proposition \ref{hom}., the sufficiency is trivial. So we can now move on to necessity.

Assume that $f$ solves equation \eq{kadd}. Then, by Proposition \ref{hom}., $f$ is $\lambda$-homogeneous on the non-negative half-line with $\lambda\in\{a,1-a\}$. Then, for $t\in\mathbb{R}$, define the function $\varphi(t):=e^{-t}f(e^t)$ and compute
\Eq{*}{
\varphi(t+\ln\lambda)
=e^{-(t+\ln\lambda)}f(e^{t+\ln\lambda})
=\frac{1}{\lambda}\cdot e^{-t}f(\lambda e^t)=\varphi(t),\qquad t\in\mathbb{R}.
}
Then, by induction, it is easy to see that $\varphi$ is periodic with any member of the set
\Eq{*}{
P:=\{n\ln a+m\ln(1-a)\mid n,m\in\mathbb{Z}\}.
}

Let $u\in\mathbb{R}$ and $\varepsilon>0$ be arbitrary but fixed and define
\Eq{*}{
x_0:=-\frac{\ln a}{\ln(1-a)}
\qquad\text{and}\qquad
y_0:=-\frac{u}{\ln(1-a)}.
}
Then, by our assumption, $x_0\in\mathbb{R}\setminus\mathbb{Q}$, thus, by Kronecker's Theorem, there exist integers $n_0,m_0\in\mathbb{Z}$ with $m_0>0$ such that
\Eq{*}{
|n_0x_0-m_0-y_0|<-\frac{\varepsilon}{\ln(1-a)}.
}
Multiplying both sides of the above inequality by $-\ln(1-a)$, we obtain that
\Eq{*}{
|n_0\ln a+m_0\ln(1-a)-u|<\varepsilon.
}
Consequently $P$ is everywhere dense in $\mathbb{R}$. By its continuity, it follows that $\varphi$ must be constant on its domain, that is, there exists $b\in\mathbb{R}$ such that, for all $t\in\mathbb{R}$, we have $f(e^t)=be^t$. Putting $x=e^t$ and taking into account the continuity of $f$, it follows that $f(x)=bx$ if $x\geq0$. By condition \eq{BT}, it also follows that $0\leq b\leq 1$.
\end{proof}


\begin{thebibliography}{30}

\bibitem{Bak90}
Baker, A., \emph{Transcendental Number Theory}, Cambridge University Press (1975).

\bibitem{For83}
Forti, G. L., \emph{On some conditional Cauchy equations on thin sets}, Boll. Un. Mat. Ital. B(6)2, 391-402 (1983).

\bibitem{For84}
Forti, G. L., \emph{Redundancy conditions for the functional equation {$f(x+h(x))=f(x)+f(h(x))$}}, Z. Anal. Anwendungen, Zeitschrift für Analysis und ihre Anwendungen 3, 549-554 (1984).

\bibitem{GlaMat25}
Głazowska, D., Matkowski, J., \emph{Weakly associative functions}, Aequat. Math. 99, 1827-1841 (2025).

\bibitem{Jar88}
Jarczyk, W., \emph{On continuous functions which are additive on their graphs}, Berichte Math.-Statist. Section in der Forschungsgesellschaft Joanneum - Graz 292 (1988).

\bibitem{Jar91}
Jarczyk, W., \emph{A recurrent method of solving iterative functional equations}, Prace Matematyczne Uniw. Si. w Katowicach 1206 (1991).

\bibitem{Mat85}
Matkowski, J., \emph{On the functional equation $\varphi(x+\varphi(x))=\varphi(x)+\varphi(\varphi(x))$}, Proceedings of the Twenty-third International Symposium on Functional Equations, Gargnano, Italy, June 2 - June 11, 1985, Centre for Information Theory, Faculty of Mathematics, University of Waterloo, Waterloo, Ontario, 24-25.

\bibitem{Mat93}
Matkowski, J., \emph{Functions which are additive on their graphs and some generalizations}, Rocznik Naukowo-Dydaktyczny., Z. 159, Prace Matematyczne 13, 233-240 (1993).

\bibitem{Mat25}
Matkowski, J., \emph{Weakly associative functions and means - new examples and open questions}, Aequat. Math. 99, 2581-2597 (2025).

\bibitem{Sab85}
Sablik, M., \emph{Note on a Cauchy conditional equation}, Rad. Mat. 1, no. 2, 241-245 (1985).

\bibitem{Zdu72}
Zdun, M., \emph{On the uniqueness of solutions of the functional equation $\phi(x+f(x))=\phi(x)+\phi(f(x))$}, Aeq. Math. 8, 229-232 (1972).

\end{thebibliography}
\end{document}